\documentclass[a4paper,11pt,leqno]{article}
\usepackage{amssymb, amsmath, amsthm, graphicx}

\newtheorem{thm}{Theorem}[section]
\newtheorem{lem}[thm]{Lemma}
\newtheorem{cor}[thm]{Corollary}
\newtheorem{prop}[thm]{Proposition}

\newtheorem{rem}{Remark}[section]

\numberwithin{equation}{section}

\renewcommand{\a}{\alpha}
\renewcommand{\b}{\beta}
\newcommand{\e}{\varepsilon}
\newcommand{\de}{\delta}
\newcommand{\fa}{\varphi}
\newcommand{\ga}{\gamma}
\renewcommand{\k}{\kappa}
\newcommand{\la}{\lambda}
\renewcommand{\th}{\theta}

\newcommand{\Si}{\Sigma}
\renewcommand{\t}{\tau}

\newcommand{\Ga}{\Gamma}
\newcommand{\La}{\Lambda}

\newcommand{\lan}{\langle}
\newcommand{\ran}{\rangle}

\def\R{{\mathbb{R}}}
\def\N{{\mathbb{N}}}
\def\Z{{\mathbb{Z}}}

\allowdisplaybreaks

\title{Equivalence of ensembles under inhomogeneous conditioning
and its applications to random Young diagrams}
\author{Tadahisa Funaki}
\date{March 30, 2011}

\begin{document}
\maketitle

\begin{abstract} 
We prove the equivalence of ensembles for Bernoulli measures on $\Z$
conditioned on two conserved quantities under the situation that one
of them is spatially inhomogeneous.  For the proof, 
we extend the classical local limit theorem for a sum of Bernoulli
independent sequences to those multiplied by linearly growing weights.
The motivation comes from the study of random Young diagrams.  We
discuss the relation between our result and the so-called Vershik
curve which appears in a scaling limit for height functions of 
two-dimensional Young diagrams.
\noindent
\footnote{
Graduate School of Mathematical Sciences,
The University of Tokyo, Komaba, Tokyo 153-8914, Japan.}
\footnote{ 
$\quad$ e-mail: funaki@ms.u-tokyo.ac.jp, 
Fax: +81-3-5465-7011.}
\footnote{
\textit{Abbreviated title $($running head$)$: Equivalence of ensembles
and applications to random Young diagrams}}
\footnote{
\textit{2010MSC: primary 60K35, secondary 82B20, 60D05, 60F05.}}
\footnote{
\textit{Supported in part by the JSPS Grants $($A$)$ 22244007
and 21654021.}}
\end{abstract}

\section{Introduction}

The equivalence of ensembles, that is the asymptotic equivalence of canonical
and grand canonical ensembles for large systems,  plays a fundamental role 
in equilibrium statistical physics \cite{R}, \cite{G}, \cite{KL} and also
in some problems related to statistics \cite{DF-1}, \cite{DF-2}.  It is mostly
discussed for canonical ensembles obtained under conditioning on spatially
homogeneous physical quantities.  In our problem, the grand canonical 
ensembles are simply Bernoulli measures on $\Z$, but the system has two
conservation laws so that the corresponding canonical ensembles
are defined through conditioning on a quantity which is not 
translation-invariant.  As a result, the macroscopic profile for the 
grand canonical ensemble turns out to be spatially dependent.  We will show
that this profile has a connection to the so-called Vershik curve which 
appears in a scaling limit for two-dimensional Young diagrams, \cite{V}, 
\cite{BBE}, \cite{FS-1}.

The height function $\psi_q$ of a two-dimensional Young diagram,
which is associated with a distinct partition $q=\{q_1>q_2>\cdots>q_K\ge 1\}$
of a positive integer $M$ by positive integers $\{q_i\}_{i=1}^K$
(i.e.\ $M=\sum_{i=1}^K q_i$), is given by a right continuous non-increasing
step function
$$
\psi_q(u) = \sum_{i=1}^K 1_{\{u<q_i\}}, \quad u\ge 0.
$$
Its height difference $\eta=\{\eta_k\}_{k\in\N}, \N = \{1,2,\ldots\},$ 
is defined by
$$
\eta_k = \psi_q(k-1)-\psi_q(k)\in\{0,1\}
\quad \text{ or } \quad
\eta_k=1 \Longleftrightarrow k\in \{q_i\}_{i=1}^K.
$$
Then, in terms of $\eta$, the height $K$ of the Young diagram at $u=0$
is represented as
$$
\psi_q(0)= \sum_{k\in\N}\eta_k
$$
and the area $M$ as 
$$
\int_0^\infty \psi_q(u)du = \sum_{k\in\N} k\eta_k.
$$
The sequence $\eta=\{\eta_k\}_{k\in\N}$ can be interpreted as defining a 
configuration of particles on $\N$ in the sense that the site $k$ is occupied 
by a particle if $\eta_k=1$ and vacant if $\eta_k=0$ and,
from the viewpoint of random Young diagrams, it is natural to restrict
the space of configurations $\eta$ to those satisfying 
$\sum_{k\in\N} \eta_k =K$ and $\sum_{k\in\N} k\eta_k=M$, given two constants 
$K$ and $M$.  The first condition is equivalent to assign the total number
of particles in the system, while the second sum involves a 
non-translation-invariant linearly growing weight $k$.  The simplest
random structure can be introduced to the set of Young diagrams with
height $K$ at $u=0$ and area $M$ by means of a uniform probability
measure on the space of configurations $\eta$ with these two constraints.
In fact, such random structure is the subject of our study.

The original motivation of this paper comes from the study of the hydrodynamic
limit for an area-preserving random dynamic on two-dimensional Young diagrams
with height differences restricted to be $0$ or $1$.  Such dynamic, which is
a kind of a surface diffusion model studied in physics \cite{PV}, is
introduced in \cite{F}.  As we have seen above, it can be transformed into 
an equivalent particle system on $\N$ with two conservation laws.

We formulate our problem and main results in Section 2.  For the proof,
we need to extend the local limit theorem for a sum of independent
random variables.  If the random variables satisfy certain proper moment
conditions (see, e.g., Theorems VII. 4 and VII.12 of \cite{Pe}), the classical
theory applies, but in our case, such conditions do not hold since
the random variables have unbounded linearly growing weights.
This is discussed in Section 3 relying on the fact that the weights are
linear by analyzing the behavior of characteristic functions.  Sections 4 
and 5 are devoted to the proofs of main results.  We finally show that
the limit curves obtained in \cite{BBE} and here are identical,
although the random structures imposed on the set of Young diagrams
are different.

\section{Main results}

Let us state our problem precisely.  We consider a particle system
on $\Z$ rather than on $\N$ following the usual setting in statistical
physics.  Each site is occupied by at most one particle.
Therefore the particle configuration on $\Z$ is represented by $\eta=
\{\eta_k\}_{k\in\Z}\in \Si :=\{0,1\}^{\Z}$, where $\eta_k=1$ means
that the site $k$ is occupied by a particle and $\eta_k=0$ means that
$k$ is vacant.  For a finite set $\La$ in $\Z$ and $\eta\in \Si$,
we define
$$
K_\La(\eta) := \sum_{k\in\La} \eta_k, \qquad
M_\La(\eta) := \sum_{k\in\La} k\eta_k.
$$
For given $K, M$ and $\ell\in\N$, let $\nu_{\La_\ell,K,M}$ be the uniform
probability measure on the configuration space 
$\Si_{\La_\ell,K,M} = \{\eta\in \Si_{\La_\ell}; K_{\La_\ell}(\eta)=K,
M_{\La_\ell}(\eta)=M\}$, where $\La_\ell=\{k\in\Z; |k|\le \ell\}$
and $\Si_\La = \{0,1\}^{\La}$.  For $\a\in (0,1)$ and
a finite set $\La$ in $\Z$, let $\nu_\a^\La$ be the Bernoulli measure on 
$\Si_\La$ with mean $\a$, that is, $\nu_\a^\La(\eta) =
(1-\a)^{\sharp\{k\in\La;\eta_k=0\}}\a^{\sharp\{k\in\La;\eta_k=1\}}$
for each $\eta\in \Si_\La$.  Then, the measure $\nu_{\La_\ell,K,M}$ 
is identical to the conditional measure of $\nu_\a^{\La_\ell}$ or $\nu_\a^\La$
with $\La_\ell\subset \La$ on $\Si_{\La_\ell,K,M}$:
For $\xi\in \Si_{\La_\ell,K,M}$,
\begin{align*}
\nu_{\La_\ell,K,M}(\xi) &= \nu_\a^{\La_\ell}(\xi|\Si_{\La_\ell,K,M})  \\
&= \nu_\a^\La(\eta_{\La_\ell}=\xi|\Si_{\La_\ell,K,M}\times
\Si_{\La\setminus\La_\ell})  \\
&= \nu_\a^\La(\eta=(\xi,\eta_{\La\setminus\La_\ell})|\Si_{\La_\ell,K,M}\times
\{\eta_{\La\setminus\La_\ell}\})
\end{align*}
independently of the choice of the outer conditions 
$\eta_{\La\setminus\La_\ell}$; see also Lemma \ref{lem:4.1} below.
We have denoted by $\eta_{\La_\ell}$ and $\eta_{\La\setminus\La_\ell}$
the configurations $\eta$ restricted on these sets.

We first observe the possible values of $K=K_{\La_\ell}(\eta)$ and 
$M=M_{\La_\ell}(\eta)$ for $\eta\in \Si_\ell$.  It is easy to
see that $0\le K\le 2\ell+1$, while the values of $M$ range as
$$
|M|\le \frac12 \big( K(2\ell+1)-K^2\big)
$$
under the condition that $K_{\La_\ell}(\eta)=K$.  Indeed the extreme
values of $M$ given $K$ are attained when all $K$ particles are closely
packed at the right or left most edges on $\La_\ell$.  Accordingly,
we have that
$$
0\le \frac{K}{2\ell+1} \le 1, \qquad
\left| \frac{M}{(2\ell+1)^2}\right| \le \frac12 \left( \frac{K}{2\ell+1}
 - \left(\frac{K}{2\ell+1}\right)^2 \right).
$$

We fix $p\in\N$ and assume that, for $\ell \in\N$ and $1\le j \le p$,
sequences $K=K_\ell$, $M=M_\ell$ and $k_j=k_{j,\ell}$ are given and satisfy
\begin{align}
&\lim_{\ell\to\infty} \frac{K}{2\ell+1} = \rho\in (0,1), 
  \label{eq:4.E.1}  \\
&\lim_{\ell\to\infty} \frac{M}{(2\ell+1)^2} = m\in (-v/2,v/2), 
  \label{eq:4.E.2}  \\
&\lim_{\ell\to\infty} \frac{k_j}{\ell} = x_j \in (-1,1), 
  \label{eq:4.E.3}
\end{align}
respectively, with distinct limits $\{x_j\}_{j=1}^p$, where $v= \rho(1-\rho)$.
A function $f$ on $\Si$ is called local if it depends only on finitely many
coordinates in $\eta$.  The shift operators $\t_i$ are defined by
$\t_i f(\eta) = f(\t_i \eta)$ and $(\t_i\eta)_k = \eta_{k+i}$ for
$i,k\in\Z$ and $\eta\in\Si$.  We are now at the position to formulate
our main theorem.

\begin{thm} \label{thm:4.E.1}
Let $f_j=f_j(\eta), 1\le j \le p$ be local functions on $\Si$.  Then, under
the conditions \eqref{eq:4.E.1}--\eqref{eq:4.E.3}, we have that
$$
\lim_{\ell\to\infty} E_{\nu_{\La_\ell,K,M}}[\prod_{j=1}^p
\t_{k_{j}} f_j] = \prod_{j=1}^p E_{\nu_{\b(x_j)}}[f_j],
$$
where $\nu_\a, \a\in (0,1)$ denotes the Bernoulli measure on $\Si$ with
mean $\a$ and 
\begin{equation} \label{eq:4.E.rho}
\b(x) \equiv \b(x;a,b) = \frac{e^{bx}a}{e^{bx}a+(1-a)},
\quad x \in [-1,1],
\end{equation}
with two parameters $a\in (0,1)$ and $b\in\R$ determined from $\rho$ and
$m$ by the relations:
\begin{equation} \label{eq:2.5-1}
\int_{-1}^1\b(x)dx = 2\rho, \quad
\int_{-1}^1 x\b(x)dx = 4m.
\end{equation}
The convergence is uniform in $(K,M)$ in the region that
$\e \le K/(2\ell+1) \le 1-\e$ and $M/(2\ell+1)^2 \in (-v/2+\e, v/2-\e)$
for every $\e>0$.
\end{thm}

For every $\rho$ and $m$, one can find $a$ and $b$ uniquely as is shown
in Lemma \ref{lem:4.2} below.
This theorem asserts that, as $\ell\to\infty$ under the canonical ensemble
$\nu_{\La_\ell,K,M}$, the limit distributions are asymptotically independent
for microscopic regions which are macroscopically separated,
and the microscopic limit distribution around the macroscopic point 
$x\in (-1,1)$ is the grand canonical ensemble $\nu_{\b(x)}$ with
macroscopically dependent profile $\b(x)$.  It will be useful to 
have another expression of $\b$:
\begin{equation} \label{eq:2.5-2}
\b(x) = \frac1b \frac{g'(x)}{g(x)} = \frac1b (\log g(x))',
\end{equation}
where $g(x) = e^{bx}a+(1-a)$.  In particular, we have that
\begin{align} \label{eq:2.a-3}
\int_{-1}^1\b(x)dx =& \frac1b \log\frac{e^ba+(1-a)}{e^{-b}a+(1-a)}, \\
\int_{-1}^1 x\b(x)dx =& \frac1b \log\frac{(e^ba+(1-a))(e^{-b}a+(1-a))}
{(1-a)^2}   \notag   \\
& + \frac1{b^2} \left\{ L_2\left(-\frac{ae^b}{1-a}\right)
- L_2\left(-\frac{ae^{-b}}{1-a}\right)\right\},    \notag
\end{align}
where $L_2(z) := -\int_0^z \frac1t \log(1-t) dt, z<1$ is the Euler
dilogarithm, see \cite{GR}, p.642.

\begin{rem}
{\rm (1)} Theorem \ref{thm:4.E.1} gives the equivalence of ensembles 
under the situation that the system has two conservation laws, especially,
one of them is not translation-invariant.  \newline
{\rm (2)} The uniformity of the convergence near the boundary values of
$K/(2\ell+1)$ and $M/(2\ell+1)^2$ can not be shown, because the cumulant
$\la_3$, which controls the error estimate in the local limit theorem, 
diverges near the boundary values.  \newline
{\rm (3)} The macroscopic mean $\b(x)$ of the limit measure is not
translation-invariant, but the distribution of the microscopic configuration
near each point $x$ is a Bernoulli measure so that it is
translation-invariant.  \newline
{\rm (4)} As we will see in the proof of Lemma \ref{lem:4.2}, $a$ is
increasing in $\rho$ and $b$ is increasing in $m$, respectively.  In 
particular, the constant $m$ measures the extent of the bias for the 
particles, that is,
a larger $m$ implies more particles on the right side. \newline
{\rm (5)} The function $\b(x)$ appears in other context: 
The asymmetric simple exclusion process on $\Z$ with jump rates
$p$ to the right and $q$ to the left satisfying $0<p<1, p+q=1$, has
$\nu_{\b(\cdot;a,b)}$ with $b=\log p/q$ as its invariant measure for
every $a\in (0,1)$, where $\nu_{\b(\cdot)}$ denotes the product measure
on $\Si$ such that $E_{\nu_{\b(\cdot)}}[\eta_k] = \b(k), k\in \Z$,
and the function $\b(\cdot;a,b)$ is defined by \eqref{eq:4.E.rho}
for all $x\in \R$ $($or $x\in\Z)$, see {\rm \cite{L}}, p.382.
\end{rem}

The macroscopic profile $\b(x)$ has a connection to the Vershik curve
which appears in a scaling limit for random Young diagrams distributed
under the restricted uniform (Fermi) statistics.  In fact,
one can associate the height function $\psi^\ell(u), u\in [-\ell-1,\ell]$
of the Young diagram with the particle configuration 
$\eta\in\Si_\ell$ by
\begin{equation} \label{eq:2.a-1}
\psi^\ell(u) = \sum_{k\in\La_\ell:k>u}\eta_k, \quad u\in [-\ell-1,\ell].
\end{equation}
Note that $\psi^\ell$ is a right continuous non-increasing step function
and satisfies
\begin{equation} \label{eq:2.5}
\psi^\ell(-\ell-1) = K_{\La_\ell}(\eta), \quad \psi^\ell(\ell)=0,
\end{equation}
with the area
\begin{equation} \label{eq:2.6}
\int_{-\ell-1}^\ell \psi^\ell(u)\, du = (\ell+1)K_{\La_\ell}(\eta)
+ M_{\La_\ell}(\eta).
\end{equation}
Under the distribution $\nu_{\La_\ell,K,M}$, we consider the macroscopically
scaled height function defined by
\begin{equation} \label{eq:2.a-2}
\tilde{\psi}^\ell(x) := \frac1{\ell}\psi^\ell(\ell x), \quad x\in [-1,1].
\end{equation}

\begin{cor} \label{cor:2.2}
Under the conditions \eqref{eq:4.E.1} and \eqref{eq:4.E.2}, $\tilde{\psi}^\ell$
converges as $\ell\to\infty$ to $\psi$ in probability
in the following sense:
$$
\lim_{\ell\to\infty} \nu_{\La_\ell,K,M}\left( \sup_{x\in [-1,1]}
\left|\tilde{\psi}^\ell(x) - \psi(x) \right| > \de \right) =0,
$$
for every $\de>0$.  The limit $\psi$ is defined by $\psi(x) = \int_x^1
\b(y)dy, x\in[-1,1]$ with $\b(x)$ determined in Theorem \ref{thm:4.E.1}.
\end{cor}

From \eqref{eq:4.E.1}, \eqref{eq:4.E.2}, \eqref{eq:2.5} and \eqref{eq:2.6},
the limit $\psi$ satisfies
\begin{equation} \label{eq:2.5-3}
\psi(-1)=2\rho, \quad \psi(1)=0, \quad \int_{-1}^1\psi(x)dx =2\rho +4m
\end{equation}
and has a slope $\psi'(x) = -\b(x)$.  Note that \eqref{eq:2.5-3} is
consistent with \eqref{eq:2.5-1}.  As we have seen in \cite{FS-1},
Theorem 2.2, the Vershik curve $\psi_R(x)$ defined for $x\ge 0$ in the 
restricted uniform statistics having the area $\int_0^\infty\psi_R(x)dx =1$
satisfies the ordinary differential equation
\begin{equation}\label{eq:Vershik}
\psi_R'' + c \psi_R'(1+\psi_R') =0,
\end{equation}
where $c=\pi/\sqrt{12}$.  Here we consider in a rectangular box, and
a simple computation with the help of \eqref{eq:2.5-2} shows that the
limit $\psi$ in Corollary \ref{cor:2.2} satisfies the same ordinary 
differential equation \eqref{eq:Vershik} with $c= -b$.
Further discussions on the Vershik curves will be held in Section 5.

\section{Local limit theorem for inhomogeneous Bernoulli sequence with
unbounded weight}

For the proof of Theorem \ref{thm:4.E.1}, we need to establish the local
limit theorem jointly for the sum of inhomogeneous Bernoulli sequence and the
sum with unbounded linearly growing weights $k$ having some defects in them.
Let a continuous function $\a=\a(\cdot): [0,1] \to (0,1)$ and a sequence
$\{\a_k^n \in (0,1)\}_{k=1}^n$ satisfying the condition
\begin{equation} \label{eq:3.alpha}
\lim_{n\to\infty} \max_{1\le k \le n} |\a_k^n -\a(k/n)| =0,
\end{equation}
be given.  Let $\{X_k\}_{k=1}^n \equiv \{X_k^n\}_{k=1}^n$
be $\{0,1\}$-valued independent random variables
with mean $E[X_k]=\a_k^n$ for $1\le k \le n$ and $n\in \N$.  Note that 
$\a_-\le \a(x), \, \a_k^n \le \a_+, x\in [0,1], 1\le k \le n, n\in \N$ holds
with some $0<\a_-<\a_+<1$.  We 
assume that a subset $\Ga_n$ of $\{1,2,\ldots,n\}$ is given for each $n\in\N$
and the size $|\Ga_n|$ is uniformly bounded in $n$: $|\Ga_n|\le C$ for all $n$.
Under these settings, we consider the sums
\begin{align}\label{eq:3.S}
S_n \equiv S_n(X) := \sum_{k\in \Ga_n^c} X_k,\quad
T_n \equiv  T_n(X) := \sum_{k\in \Ga_n^c} kX_k,
\end{align}
for $n\in\N$, where $\Ga_n^c = \{1,2,\ldots,n\} \setminus \Ga_n$;
$\Ga_n$ are defects in these sums.  Then, it is easy to see that the (joint)
central limit theorem holds for $(S_n,T_n)$.  Indeed, we define
\begin{align*}
\tilde{S}_n := \frac1{\sqrt{U_n}} \big(S_n - E_n \big), \quad
\tilde{T}_n := \frac1{\sqrt{V_n}} \big(T_n - F_n \big),
\end{align*}
with
\begin{align*}
E_n = \sum_{k\in \Ga_n^c} \a_k^n, \quad
F_n = \sum_{k\in \Ga_n^c} k\a_k^n, \quad
U_n = \sum_{k\in \Ga_n^c} v_k^n, \quad
V_n = \sum_{k\in \Ga_n^c} k^2 v_k^n,
\end{align*}
where $v_k^n :=\a_k^n(1-\a_k^n)$.  Recalling that $\a(\cdot)\in C([0,1])$,
the condition \eqref{eq:3.alpha} and $|\Ga_n|\le C$, the asymptotic behaviors
of $E_n, F_n, U_n$ and $V_n$ as $n\to\infty$ are given by
\begin{align} \label{eq:3.0}
&  E_n=n(\bar{\a}+o(1)), \quad F_n=n^2(\check{\a}+o(1)), \\
&  U_n=n(\bar{v}+o(1)), \quad V_n=n^3(\check{v}+o(1)),  \notag
\end{align}
where $\bar{\a}, \check{\a}, \bar{v}, \check{v}$ are positive constants
defined by
\begin{align*}
& \bar{\a}= \int_0^1\a(x)dx, \quad
\check{\a}= \int_0^1 x\a(x)dx, \\
& \bar{v}= \int_0^1\a(x)(1-\a(x))dx, \quad
\check{v}=\int_0^1x^2\a(x)(1-\a(x))dx,
\end{align*}
respectively.  Then $(\tilde{S}_n,\tilde{T}_n)$ weakly converges to 
$(Y_1,Y_2)$ as $n\to\infty$, where $Y=(Y_1,Y_2)$ is an $\R^2$-valued 
Gaussian random variable with mean $0$ and $E[Y_1^2]=E[Y_2^2]=1$, 
$E[Y_1Y_2]= \la$, where
$$
\la = \frac1{\sqrt{\bar{v}\check{v}}} \int_0^1 x\a(x)(1-\a(x)) dx.
$$
Indeed, the convergence of the corresponding characteristic functions
is shown by Lemma \ref{lem:6.1} below.  
Note that $\la\in (0,1)$ by Schwarz's inequality.
The joint distribution density function of $Y$ is given by
\begin{equation} \label{eq:4.E.4}
q_0(y) = \frac1{2\pi\sqrt{1-\la^2}} \exp \left\{ - 
\frac{y_1^2 - 2\la y_1y_2+y_2^2}{2(1-\la^2)} \right\}, \quad 
y = (y_1,y_2) \in \R^2.
\end{equation}

We now state the corresponding local limit theorem.  The set of all
possible values of $(S_n,T_n)$ is denoted by
$$
\mathcal{P}_n := \{(K,L) \in \Z_+\times\Z_+; P(S_n=K, T_n=L)>0\},
$$
where $\Z_+ = \{0,1,2,\ldots\}$.

\begin{prop} \label{lem:4.E.2}
We have that
$$
\lim_{n\to\infty} \sup_{(K,L)\in \mathcal{P}_n} \left| \sqrt{U_nV_n} 
\, P(S_n=K, T_n=L) - q_0(y_1,y_2) \right| =0,
$$
where
\begin{equation} \label{eq:4.E.5}
y_1 = \frac1{\sqrt{U_n}} (K-E_n), \quad
y_2 = \frac1{\sqrt{V_n}} (L-F_n).
\end{equation}
\end{prop}

\begin{rem}
If $\a(\cdot)\in C^1([0,1])$ and $\a_k^n$ are given by $\a_k^n = \a(k/n)$,
then the convergence in Proposition \ref{lem:4.E.2} takes place with speed
$O(1/\sqrt{n})$:
$$
\sup_{(K,L)\in \mathcal{P}_n} \left| \sqrt{U_nV_n} 
\, P(S_n=K, T_n=L) - q_0(y_1,y_2) \right| \le \frac{C}{\sqrt{n}},
$$
with some $C>0$.  However, in Section 4, we are forced to consider more general
$\a_k^n$ satisfying the condition \eqref{eq:3.alpha}. 
See Remark \ref{rem:4.1} below.
\end{rem}

The local limit theorem for $T_n$ was studied by \cite{H} in homogeneous
Bernoulli case without defects and we extend it to the joint variables
$(S_n,T_n)$ in inhomogeneous case with defects. The rest of this section
is devoted to the proof of Proposition \ref{lem:4.E.2}.  We essentially 
follow the arguments in \cite{Pe}, but because of the unboundedness of 
the weights $k$ for $T_n$, a phenomenon different from the classical 
situation arises in the Fourier mode especially for the term $I_2$ 
introduced below.

Let $f(s,t;S_n,T_n), s, t \in \R,$ be the characteristic
function of $\R^2$-valued random variable $(S_n,T_n)$:
$$
f(s,t;S_n,T_n) = E\left[ e^{i(sS_n+tT_n)}\right],
$$
where $i=\sqrt{-1}$.  Then, for all $(K,L)\in \mathcal{P}_n$, we have that
\begin{align*}
(2\pi)^2 P(S_n=K, T_n=L)
& = \int_{-\pi}^\pi \int_{-\pi}^\pi e^{-i(sK+tL)} f(s,t;S_n,T_n)\, dsdt \\
& = \frac1{\sqrt{U_nV_n}} \int_{-\sqrt{U_n}\pi}^{\sqrt{U_n}\pi}
 \int_{-\sqrt{V_n}\pi}^{\sqrt{V_n}\pi} e^{-i(sy_1+ty_2)} 
 f(s,t;\tilde{S}_n,\tilde{T}_n)\, dsdt,
\end{align*}
where $y_1, y_2$ are defined by \eqref{eq:4.E.5}.  The second equality
is due to a change of variables noting that
\begin{equation} \label{eq:6.0}
f(s,t;S_n,T_n)
= \exp\{ i (sE_n + tF_n)\}
f(\sqrt{U_n}s,\sqrt{V_n}t;\tilde{S}_n,\tilde{T}_n).
\end{equation}
Thus, if we define the error term by
$$
R_n(K,L) = (2\pi)^2 \{\sqrt{U_nV_n} \, P(S_n=K, T_n=L) - q_0(y_1,y_2)\},
$$
it can be rewritten as
\begin{align*}
R_n(K,L) =& \int_{-\sqrt{U_n}\pi}^{\sqrt{U_n}\pi}
 \int_{-\sqrt{V_n}\pi}^{\sqrt{V_n}\pi} e^{-i(sy_1+ty_2)} 
 f(s,t;\tilde{S}_n,\tilde{T}_n)\, dsdt  \\
 & - \int_{\R^2} e^{-i(sy_1+ty_2)} e^{-(s^2+2\la st+t^2)/2} \, dsdt,
\end{align*}
since $(2\pi)^2 q_0(y_1,y_2)$ is given by the second integral in the above
expression.  We divide $R_n(K,L)$ into the sum of three integrals:
$$
R_n(K,L) = I_1+I_2+I_3,
$$
where
\begin{align*}
I_1 & = \int_{D_{1,n}} e^{-i(sy_1+ty_2)} \left\{
 f(s,t;\tilde{S}_n,\tilde{T}_n) -  e^{-(s^2+2\la st+t^2)/2} 
 \right\} dsdt, \\
I_2 & = \int_{D_{2,n}} e^{-i(sy_1+ty_2)} f(s,t;\tilde{S}_n,\tilde{T}_n)
  \, dsdt,  \\
I_3 & = \int_{D_{3,n}} e^{-i(sy_1+ty_2)} 
 e^{-(s^2+2\la st+t^2)/2} \, dsdt,
\end{align*}
respectively.  Three domains are defined by
\begin{align*}
D_{1,n} & = \{(s,t)\in\R^2; |s|, |t| \le c\sqrt{n}\}, \\
D_{2,n} & = \{(s,t)\in\R^2; c\sqrt{n}<|s| \le \sqrt{U_n}\pi
\; \text{ or }\; c\sqrt{n}<|t| \le \sqrt{V_n}\pi\}, \\
D_{3,n} & = \{(s,t)\in\R^2; |s|\ge c\sqrt{n} \;\text{ or }\;
 |t| \ge c\sqrt{n}\},
\end{align*}
respectively, with a small enough $c \in (0, \sqrt{\bar{v}}\pi)$ chosen later.

The estimate on $I_3 \equiv I_{3,n}(K,L)$ is easy.  In fact, noting that
$s^2+2\la st+t^2 \ge (1-\la)(s^2+t^2)$ and $\la<1$, we have a uniform bound
on $I_3$:  There exist $C_1, c_1>0$ such that
\begin{equation} \label{eq:6.1}
\sup_{K,L} |I_3| \le C_1 e^{-c_1 n}.
\end{equation}

To give the estimate on $I_1$, we prepare a lemma which is a 
two-dimensional version of Lemma 1 in Chapter V of \cite{Pe}, p.109.

\begin{lem} \label{lem:6.1}
For every $\de>0$, there exist $n_0 \in\N$, $c_2, c_3>0$ such that,
if $n\ge n_0$, 
$$
\left|
 f(s,t;\tilde{S}_n,\tilde{T}_n) -  e^{-(s^2+2\la st+t^2)/2}  \right|
\le \de (|s|^3+|t|^3+s^2+t^2) e^{-c_2(s^2+t^2)}
$$
holds for every $(s,t)\in \R^2$ satisfying $|s|, |t| \le c_3\sqrt{n}$.
\end{lem}

\begin{proof}
A simple computation leads to
\begin{equation} \label{eq:6.2}
f(s,t;\tilde{S}_n,\tilde{T}_n) = \prod_{k\in \Ga_n^c}
 \left\{\a_k e^{i\ga_k(1-\a_k)}+(1-\a_k) e^{-i\ga_k\a_k}\right\},
\end{equation}
where $\ga_k \equiv \ga_k^n(s,t) := s/\sqrt{U_n} +tk/\sqrt{V_n}$ and
we simply write $\a_k$ instead of $\a_k^n$.
If $|s|, |t| \le c_3\sqrt{n}$ with $c_3>0$ chosen later, from \eqref{eq:3.0},
$\ga_k$ can be estimated as
\begin{equation} \label{eq:6.3}
|\ga_k| \le \bar{c} c_3,
\end{equation}
with some $\bar{c}>0$.  Therefore,
since Taylor's formula implies
$$
\left| e^z - \left(1+z+\frac12 z^2\right)\right| \le C_4 |z|^3,
$$
with some $C_4>0$ for all $z\in\mathbb{C}: |z| \le \bar{c}c_3$,
we have that
\begin{align} \label{eq:6.4}
& \a_k e^{i\ga_k(1-\a_k)}+(1-\a_k) e^{-i\ga_k\a_k} \\
& \quad = \a_k \big( 1 + i\ga_k(1-\a_k) -\frac12 \ga_k^2(1-\a_k)^2\big)
 + (1-\a_k) \big( 1 - i\ga_k\a_k -\frac12 \ga_k^2\a_k^2\big) + R_k \notag  \\
& \quad = 1- \frac{v_k}2 \ga_k^2 + R_k,  \notag
\end{align}
with an error term $R_k \in\mathbb{C}$ having an estimate:
\begin{equation} \label{eq:6.5}
|R_k| \le C_4 \big(\a_k(1-\a_k)^3 |\ga_k|^3 + (1-\a_k)\a_k^3 |\ga_k|^3 \big)
\le C_5 |\ga_k|^3,
\end{equation}
where we write  $v_k$ for $v_k^n$.
Let $\log(1+z)$ be the principal value defined for $z\in \mathbb{C}
\setminus \{z=x\in\R; x\le -1\}$.  Then,
$$
|\log(1+z) -z| \le C_6 |z|^2
$$
holds for $|z|< 1/2$ with some $C_6>0$, and therefore, from \eqref{eq:6.2}
and \eqref{eq:6.4}, we have that
\begin{equation} \label{eq:6.6}
\left| \log f(s,t;\tilde{S}_n,\tilde{T}_n) - \sum_{k\in\Ga_n^c}
  \big(-\frac{v_k}2 \ga_k^2 + R_k\big) \right| \le C_6 \sum_{k\in\Ga_n^c}
    \left|-\frac{v_k}2 \ga_k^2 + R_k \right|^2,
\end{equation}
if $|-v_k \ga_k^2/2 + R_k |<1/2$ for all $1\le k \le n$.
Note that, from \eqref{eq:6.3} and \eqref{eq:6.5}, the last condition
holds if we choose $c_3>0$ sufficiently small.  However, we see that
\begin{equation} \label{eq:3.a}
\sum_{k\in\Ga_n^c} v_k \ga_k^2 
=  (s^2+2\la st+t^2)(1+o(1)),
\end{equation}
by using \eqref{eq:3.0} and recalling $|\Ga_n|\le C$.  We thus obtain that
\begin{equation} \label{eq:6.7}
\log f(s,t;\tilde{S}_n,\tilde{T}_n) = -\frac12 (s^2+2\la st+t^2) +R,
\end{equation}
with an error term $R$ having a bound:
$$
|R| \le \sum_{k=1}^n|R_k| + 2C_6 \left( \frac14 \sum_{k=1}^nv_k^2|\ga_k|^4
+ \sum_{k=1}^n|R_k|^2\right) + o(1) (s^2+t^2).
$$
The sum over $\Ga_n^c$ is bounded by that over all $1\le k \le n$.
However, from \eqref{eq:6.5} and \eqref{eq:3.0}, we have that
\begin{equation} \label{eq:6.8}
\sum_{k=1}^n|R_k| \le C_5 \sum_{k=1}^n|\ga_k|^3
\le \frac{C_7}{\sqrt{n}} (|s|^3+|t|^3),
\end{equation}
with some $C_7>0$.  For the sum of $v_k^2|\ga_k|^4$, since $v_k^2\le (1/4)^2$,
we have that
\begin{align} \label{eq:6.9}
\sum_{k=1}^nv_k^2|\ga_k|^4 & \le C_8 \sum_{k=1}^n
  \left(\frac{|s|^4}{U_n^2} + \frac{|t|^4k^4}{V_n^2} \right) \\
& \le \frac{C_{9}}{n} (|s|^4+|t|^4)
\le \frac{C_{9}c_3}{\sqrt{n}} (|s|^3+|t|^3),  \notag
\end{align}
with some $C_8, C_{9}>0$, if $|s|, |t| \le c_3\sqrt{n}$.  Since \eqref{eq:6.3}
and \eqref{eq:6.5} show that $|R_k|$ is bounded, \eqref{eq:6.8} implies that
\begin{equation} \label{eq:6.10}
\sum_{k=1}^n|R_k|^2\le \frac{C_{10}}{\sqrt{n}} (|s|^3+|t|^3).
\end{equation}
Therefore, \eqref{eq:6.8}--\eqref{eq:6.10} can be summarized into
\begin{align*}
|R| \le \frac{C_{11}}{\sqrt{n}} (|s|^3+|t|^3) + o(1) (s^2+t^2),
\end{align*}
if $|s|, |t| \le c_3\sqrt{n}$.
Coming back to \eqref{eq:6.7}, we have that
\begin{align*}
\left|
 f(s,t;\tilde{S}_n,\tilde{T}_n) -  e^{-(s^2+2\la st+t^2)/2}  \right|
& = e^{-(s^2+2\la st+t^2)/2}  |e^R-1|  \\
& \le e^{-(1-\la)(s^2+t^2)/2} |R| |e^{|R|}.
\end{align*}
However, if $|s|, |t| \le c_3\sqrt{n}$, $|R| \le (C_{11}c_3+o(1)) (s^2+t^2)$
and therefore, by choosing $c_3>0$ sufficiently small and $n_0$ sufficiently
large, we have that
$$
e^{-(1-\la)(s^2+t^2)/2} e^{|R|} \le C_{12} e^{-c_2(s^2+t^2)},
$$
with some $c_2, C_{12}>0$ for every $n\ge n_0$. We have thus completed the
proof of the lemma by changing the choice of $n_0$ to bound $|R|$ by means
of the given $\de>0$, if necessary.
\end{proof}

Lemma \ref{lem:6.1} gives a uniform estimate on $I_1\equiv I_{1,n}(K,L)$
under the choice of $c=c_3$: For every $\de>0$,
\begin{equation} \label{eq:6.11}
\sup_{K,L} |I_1| \le \de \int_{\R^2} (|s|^3+|t|^3 +s^2+t^2)
 e^{-c_2(s^2+t^2)} \, dsdt,
\end{equation}
holds if $n\ge n_0$; note that the last integral is converging.

Finally, we give an estimate on $I_2\equiv I_{2,n}(K,L)$.
Using the relation \eqref{eq:6.0}, $I_2$ can be rewritten and estimated as
\begin{equation} \label{eq:6.12}
|I_2| \le \sqrt{U_nV_n} \int_{E_{2,n}} |f(s,t;S_n,T_n)| \, dsdt,
\end{equation}
where $E_{2,n} = \{(s,t); c_4\le |s| \le \pi \;\text{ or }\;
c_5/n\le |t| \le \pi\}$ and $c_4=c\, \inf_{n\in\N}\sqrt{n/U_n}$,
$c_5=c\, \inf_{n\in\N}\sqrt{n^3/V_n} >0$.  Once we can show that there
exist sufficiently small $\th>0$ and $\k\in (0,1)$ such that
\begin{equation} \label{eq:6.13}
\sharp \{k; 1\le k \le n, k\notin \Ga_n, s+kt\notin [-\th,\th] \;\text{ mod }
\; 2\pi\}\ge \k n-1 -|\Ga_n|,
\end{equation}
for all $(s,t)\in E_{2,n}$, then we have that
$$
|f(s,t;S_n,T_n)| = \left|\prod_{k\in \Ga_n^c} \{\a_k e^{i(s+kt)} 
+ (1-\a_k)\}\right| \le r^{\k n-1-C},
$$
where $\a_k = \a_k^n$ and
$$
r := \max_{u: \th\le |u|\le \pi} \max_{k\in \Ga_n^c}|\a_k e^{iu}+(1-\a_k)| < 1,
$$
by recalling $0<\a_-\le \a_k\le \a_+<1$.  This together with 
\eqref{eq:6.12} and \eqref{eq:3.0} proves
\begin{align} \label{eq:6.14}
\sup_{K,L} |I_2| & \le (2\pi)^2 \sqrt{U_n V_n} r^{\k n-1-C} \\
& \le C' n^2 r^{\k n-1-C}.  \notag
\end{align}
Three uniform estimates \eqref{eq:6.1}, \eqref{eq:6.11} and \eqref{eq:6.14}
conclude the proof of Proposition \ref{lem:4.E.2}.

The final task is to show \eqref{eq:6.13} for all $(s,t)\in E_{2,n}$.  
To this end, we may assume $t\ge 0$ by symmetry.  We may also assume
$\Ga_n=\emptyset$ and prove \eqref{eq:6.13} without $-|\Ga_n|$ in the
right hand side.  We rewrite the
region $E_{2,n}\cap\{t\ge 0\}$ into a union of three regions: 
$E_{2,n}\cap\{t\ge 0\} = E_{2,n}^{(1)}\cup E_{2,n}^{(2)}\cup E_{2,n}^{(3)}$,
where
\begin{align*}
E_{2,n}^{(1)} & = \{(s,t); 2\th \le t \le \pi, |s|\le \pi\},\\
E_{2,n}^{(2)} & = \{(s,t); \frac{c_5}n\le t<2\th, |s|\le \pi\},\\
E_{2,n}^{(3)} & = \{(s,t); 0\le t <\frac{c_5}n, c_4\le |s|\le \pi\}.
\end{align*}
Note that $c (=c_3)$ and therefore $c_5$ was chosen
sufficiently small so that we may assume $0<c_5<2\pi$.  For $(s,t) \in
E_{2,n}^{(1)}$, since $\lq\lq s+kt\in [-\th,\th]$ mod $2\pi$" implies
$\lq\lq s+(k+1)t\notin [-\th,\th]$ mod $2\pi$", it is obvious that
\eqref{eq:6.13} holds for $\Ga_n=\emptyset$ and $\k = 1/2$.
Now we take $(s,t) \in E_{2,n}^{(2)}$.  The $n$ points $\{kt\}_{n=1}^n$
are arranged on $\R$ in an equal distance and the interval $[t,nt]$
containing all these points are covered by at most $m:= [(n-1)t/(2\pi)]+1$
disjoint intervals of length $2\pi$, where $[\;\;]$ means the integer part.
However, for an arbitrary interval $I\subset \R$ of length $2\pi$,
$$
\sharp\{k; 1\le k \le n, kt\in I, s+kt\in [-\th,\th] \;\text{ mod }\; 2\pi\}
\le \frac{2\th}t +1.
$$
Thus, since $c_5/n\le t <2\th$ for $(s,t) \in E_{2,n}^{(2)}$, we have that
\begin{align*}
& \sharp\{k; 1\le k \le n, s+kt\in [-\th,\th] \;\text{ mod }\; 2\pi\}
 \le m \left(\frac{2\th}t +1 \right)
\le \left(\frac{nt}{2\pi} +1 \right)\left(\frac{2\th}t +1 \right) \\
& \qquad\qquad = \frac{\th}{\pi} n + \frac{t}{2\pi}n + \frac{2\th}t +1
\le \frac{2\th}{\pi} n + \frac{2\th}{c_5}n +1
\le \frac{n}2+1,
\end{align*}
by choosing $\th: 0<\th<(1/\pi+1/c_5)^{-1}/4$, and this proves \eqref{eq:6.13}
for $\Ga_n=\emptyset$ and $\k = 1/2$.
Finally we take $(s,t) \in E_{2,n}^{(3)}$.  Then, since $0<c_5<2\pi$
and $0\le t \le c_5/n$, $n$ points $\{kt\}_{k=1}^n$ are all located in the
interval $[0,2\pi)$.  Now choose $\th: 0<\th<c_4/8$.  Then, recalling that
$c_4\le |s| \le \pi$, for example in the case that $-\pi\le s \le -c_4$,
\begin{align*}
& \{k; 1\le k \le n, s+kt\in [-\th,\th] \;\text{ mod }\; 2\pi\}
 = \{k; 1\le k \le n, s+kt\in [-\th,\th] \}  \\
& \qquad\qquad \subset \{k; 1\le k \le n, kt\ge -\th-s\}
\subset \{k; 1\le k \le n, k\ge \frac{7c_4}{8c_5}n \}.
\end{align*}
Thus we obtain \eqref{eq:6.13} for $\Ga_n=\emptyset$ by choosing $\k =
(7c_4/(8c_5))\wedge 1>0$.
The case that $c_4\le s \le \pi$ can be discussed in a similar way. The 
proof of \eqref{eq:6.13} is completed for all $(s,t)\in E_{2,n}$. 

\begin{rem}
Our argument relies on the specific form $k$ of the weights.  It can be
extended to linearly growing weights, but not for general weights such
as a power of $k$.
\end{rem}

\section{Proof of Theorem \ref{thm:4.E.1}}

For a function $\b=\b(\cdot): [-1,1] \to (0,1)$, we denote by 
$\nu_{\b(\cdot)}^{\La_\ell}$ the distribution on $\Si_{\La_\ell}$ of 
$\{0,1\}$-valued independent sequences $\{\eta_k\}_{k\in\La_\ell}$ such that
$E[\eta_k]=\b(k/\ell), k\in \La_\ell$.  The next lemma explains the reason
that the functions of the form \eqref{eq:4.E.rho} appear in the limit.

\begin{lem} \label{lem:4.1}
For a function $\b(\cdot) \equiv \b(\cdot;a,b)$ of the form
\eqref{eq:4.E.rho}, the conditional measure of $\nu_{\b(\cdot)}^{\La_\ell}$
on $\Si_{\La_\ell,K,M}$ is a uniform probability measure for every
$a\in (0,1)$ and $b\in\R$:
$$
\nu_{\b(\cdot)}^{\La_\ell}(\cdot|\Si_{\La_\ell,K,M})=
\nu_{\La_\ell,K,M}(\cdot).
$$
\end{lem}

\begin{proof}
For $\a\in (0,1)$, let $\mu_\a$ be the probability measure on $\{0,1\}$
defined by $\mu_\a(1)=\a$.  Then, if $\a= e^ba/(e^ba+(1-a))$ for some
$a\in (0,1)$ and $b\in\R$, it holds that
$$
\mu_\a(\xi) = z_{a,b}^{-1} e^{b\xi}\mu_a(\xi),
$$
for $\xi=0,1$ with a constant $z_{a,b} = e^ba+(1-a)$.
Therefore, for $\nu_{\b(\cdot)}^{\La_\ell}$ on $\Si_{\La_\ell}$ with 
$\b(\cdot)$ of the form \eqref{eq:4.E.rho}, we have that
\begin{align*}
\nu_{\b(\cdot)}^{\La_\ell}(\eta) 
& = \prod_{k\in\La_\ell} \mu_{\b(k/\ell)}(\eta_k) 
 = \prod_{k\in\La_\ell} z_{a, bk/\ell}^{-1} 
e^{bk\eta_k/\ell} \mu_a(\eta_k) \\
& = Z^{-1} e^{b\sum_{k\in\La_\ell} k\eta_k/\ell} \nu_a^{\La_\ell}(\eta),
\end{align*}
for $\eta\in\Si_{\La_\ell}$ with a normalizing constant $Z \equiv 
\prod_{k\in\La_\ell} z_{a, bk/\ell}= E_{\nu_a^{\La_\ell}}[
e^{b\sum_{k\in\La_\ell} k\eta_k/\ell}]$.  This implies that
$$
\nu_{\b(\cdot)}^{\La_\ell}(\cdot|\Si_{\La_\ell,K,M})
=\nu_a^{\La_\ell}(\cdot|\Si_{\La_\ell,K,M}),
$$
since $\sum_{k\in\La_\ell} k\eta_k =M$ on $\Si_{\La_\ell,K,M}$.
However, $\nu_a^{\La_\ell}(\cdot|K_{\La_\ell}=K)$ is a uniform measure on
$\{\eta\in \Si_{\La_\ell}; K_{\La_\ell}(\eta)=K\}$ so that
$\nu_a^{\La_\ell}(\cdot|\Si_{\La_\ell,K,M})$ is also a uniform measure on
$\Si_{\La_\ell,K,M}$, that is, $\nu_{\La_\ell,K,M}$.
\end{proof}

We next establish the one-to-one correspondence between $(\rho,m)$ and
$(a,b)$ defined by \eqref{eq:2.5-1}.  In particular, one can uniquely find
$(a,b)$ for every $(\rho,m)$ in Theorem \ref{thm:4.E.1}.
Consider the map $\Phi$ for $(a,b)\in (0,1)\times\R$ to $(\rho,m)=
(F(a,b),G(a,b))\in D:= \{(\rho,m); \rho\in (0,1), m\in (-v/2,v/2)\}$
with $v=\rho(1-\rho)$ defined by
$$
F(a,b) = \frac12\int_{-1}^1\b(x; a,b)dx, \quad
G(a,b) = \frac14\int_{-1}^1 x\b(x;a,b)dx.
$$

\begin{lem} \label{lem:4.2}
The map $\Phi$ is a diffeomorphism from $(0,1)\times\R$ onto the domain $D$.
\end{lem}

\begin{proof}
Recalling that $\b(x;a,b)$ is given by $\b(x;a,b) = e^{bx}a/g(x;a,b)$
with $g(x)\equiv g(x;a,b)= e^{bx}a+(1-a)$, we can easily compute the 
derivatives of $F$ and $G$:
\begin{align*}
& \frac{\partial F}{\partial a} 
= \frac12\int_{-1}^1 \frac{e^{bx}}{g(x)^2} dx,
\quad 
\frac{\partial F}{\partial b} 
= \frac{a(1-a)}2\int_{-1}^1 \frac{xe^{bx}}{g(x)^2} dx, \\
& \frac{\partial G}{\partial a} 
= \frac14\int_{-1}^1 \frac{xe^{bx}}{g(x)^2} dx,
\quad 
\frac{\partial G}{\partial b} 
= \frac{a(1-a)}4\int_{-1}^1 \frac{x^2e^{bx}}{g(x)^2} dx.
\end{align*}
This, with the help of Schwarz's inequality, implies that the Jacobian
of the map $\Phi$ is positive everywhere, that is, 
$J= \frac{\partial F}{\partial a}\frac{\partial G}{\partial b}
- \frac{\partial F}{\partial b}\frac{\partial G}{\partial a} >0$.
In particular, the map $\Phi$ is a local diffeomorphism.

For every $b\in\R$, set $C_b := \{(F(a,b),G(a,b))\in D; a\in (0,1)\}$.
Then, $C_b$ is a Jordan arc in $D$ connecting two points $(0,0)$ and $(1,0)$.
Indeed, from $\partial F/\partial a>0$, $C_b$ has no double points so that
it is a Jordan arc.  As $a\downarrow 0$, we have $\b(x;a,b)\to 0$ so that
$(F(a,b),G(a,b))\to (0,0)$, while $(F(a,b),G(a,b))\to (1,0)$ as
$a\uparrow 1$, since $\b(x;a,b)\to 1$.  Especially, 
$C_0=\{(a,0);a\in (0,1)\}$ is a line segment connecting $(0,0)$ and $(1,0)$.
From the fact that $J>0, \partial F/\partial a>0$ and 
$\partial G/\partial b>0$, we see that the arc $C_{b_1}$ is located above
$C_{b_2}$ in the $\rho$-$m$ plane if $b_1>b_2$.  This proves that the map
$\Phi$ is one to one.

To show the onto property of $\Phi$, we consider $(a,b)$ satisfying
$F(a,b) = \rho$ for a fixed $\rho\in (0,1)$.  From \eqref{eq:2.a-3},
this condition can be rewritten as
$$
a = \frac{e^{2b\rho}-1}{e^b+ e^{2b\rho}- e^{b(2\rho-1)}-1}
$$
and therefore
$$
\b(x;a,b) = \frac{e^{bx}(e^{2b\rho}-1)}{e^{bx}(e^{2b\rho}-1)+
(e^b - e^{b(2\rho-1)})},
$$
which behaves as
\begin{equation*}
\b(x;a,b) \to \left\{ \begin{aligned}
1 \quad & \text{ if } \quad x+2\rho>1, \\
0 \quad & \text{ if } \quad x+2\rho<1, 
\end{aligned} \right.
\end{equation*}
as $b\to\infty$.  Thus we have
$$
\lim_{b\to\infty} G(a,b) = \frac14 \int_{1-2\rho}^1 xdx
= \frac12 \rho(1-\rho),
$$
under the condition $F(a,b)=\rho$.  In a similar way, we can show that
$$
\lim_{b\to-\infty} G(a,b) = -\frac12 \rho(1-\rho).
$$
This proves the onto property of the map $\Phi$.
\end{proof}

We are now at the position to complete the proof of Theorem \ref{thm:4.E.1}.
Our method is standard in the sense that we apply the local limit theorem
to compute the conditional distributions, see, e.g., \cite{KL}, p.353.
The novelty lies in the following point.  When applying Proposition 
\ref{lem:4.E.2}, the term $q_0(y_1,y_2)$ needs to be uniformly positive to
dominate the error term and this restricts our choice of the 
underlying measures in such a way that both $y_1$ and $y_2$ are sufficiently
close to $0$.  For instance, in \cite{KL}, it was only required to adjust 
the density parameter so that the macroscopic profile was constant over the
space, but here the problem involves two parameters $K$ and $M$ or $\rho$ 
and $m$.  However, once we take the inhomogeneous Bernoulli measures with
the functions $\b(\cdot; a,b)$ as macroscopic profiles, our choice allows 
two parameters $a$ and $b$.  This suits our purpose and we can realize the
situation that both $y_1$ and $y_2$ are close to $0$ at the same time under 
a suitable choice.  The local limit theorem with defects is prepared to treat
the numerator $B_\ell$ introduced later.

We denote the supports of the local functions $f_j$ in Theorem \ref{thm:4.E.1}
by $\Ga_j\subset\Z$ and set $\Ga (=\Ga^{(\ell)})= \cup_{j=1}^p \t_{k_j}\Ga_j$, 
where $\t_{k_j}\Ga_j = \Ga_j+k_j$ and recall $k_j = k_{j,\ell}$.  
Note that $\Ga\subset \La_\ell$ and
$\{\t_{k_j}\Ga_j\}_{j=1}^p$ are disjoint if $\ell$ is sufficiently large,
since $k_j$ asymptotically behave as $x_j\ell$ and $\{x_j\}_{j=1}^p \subset
(-1,1)$ are distinct.  Then, for every function $\b_\ell(\cdot)$ of the form
\eqref{eq:4.E.rho}, Lemma \ref{lem:4.1} shows that
\begin{align} \label{eq:4.E.8}
& E_{\nu_{\La_\ell,K,M}}[\prod_{j=1}^p \t_{k_{j}} f_j] 
- \prod_{j=1}^p E_{\nu_{\b(x_j)}}[f_j] \\
& = \sum_{\xi\in\{0,1\}^\Ga} \big(\prod_{j=1}^p f_j(\xi_j)
- \prod_{j=1}^p E_{\nu_{\b(x_j)}}[f_j]\big) 
\frac{\nu_{\b_\ell(\cdot)}^{\La_\ell}\big( \eta|_\Ga = \xi, 
K_{\La_\ell}(\eta)=K, M_{\La_\ell}(\eta)=M\big)}
{\nu_{\b_\ell(\cdot)}^{\La_\ell}\big( K_{\La_\ell}(\eta)=K,
M_{\La_\ell}(\eta)=M\big)},  \notag
\end{align}
where $\b(\cdot)$ is the function determined in Theorem \ref{thm:4.E.1},
$\eta|_\Ga$ stands for the restriction of $\eta$ to $\Ga$ and
we denote by $\xi_j = \xi|_{\Ga_j}$.  For given $K=K_\ell$ and $M=M_\ell$,
our special choice of $\b_\ell(\cdot)$ will be the function 
$\b_\ell(\cdot)=\b(\cdot;a_\ell,b_\ell)$ of the form \eqref{eq:4.E.rho}
with the solution $(a_\ell,b_\ell)$ of two equations \eqref{eq:2.5-1}
for $\rho=K/(2\ell+1)$ and $m=M/(2\ell+1)^2$.

To apply Proposition \ref{lem:4.E.2} for $\{\eta_k\}_{k\in\La_\ell}$, we need
to shift it and consider $X=\{X_k\}_{k=1}^n$ with $n=2\ell+1$ 
determined by $X_k = \eta_{k-\ell-1} (\equiv (\t_{\ell+1}^{-1}\eta)_k)$ for
$k=1,2,\ldots,n$.  For such $X$ and defects $\Ga_n \subset \{1,2,\ldots,n\}$,
we denote two sums $S_n$ and $T_n$ in \eqref{eq:3.S} by $S_n^{\Ga_n}(X)$ and
$T_n^{\Ga_n}(X)$, respectively, to indicate the defects clearly.  Then, 
we have that 
\begin{equation} \label{eq:4.E.7}
S_n^\emptyset(X) = K_{\La_\ell}(\eta)
\quad \text{ and } \quad
T_n^\emptyset(X) = M_{\La_\ell}(\eta) + (\ell+1) K_{\La_\ell}(\eta).
\end{equation}
The numerator of the fractional expression in the right
hand side of \eqref{eq:4.E.8} is equal to
\begin{align*}
& P_{\a^n}\big( X|_{\t_{\ell+1}\Ga} = \t_{\ell+1}^{-1}\xi, 
  S_n^\emptyset(X)=K, T_n^\emptyset(X)=M+ (\ell+1) K\big)  \\
& \quad = \nu_{\b_\ell(\cdot)}^\Ga(\xi) P_{\a^n}\left( 
S_n^{\t_{\ell+1}\Ga}(X)
  =K - S_n^{(\t_{\ell+1}\Ga)^c}(\t_{\ell+1}^{-1}\xi), \right.\\
& \qquad\qquad\qquad\qquad  \left.
T_n^{\t_{\ell+1}\Ga}(X) =M+ (\ell+1) K - 
   T_n^{(\t_{\ell+1}\Ga)^c}(\t_{\ell+1}^{-1}\xi)\right),
\end{align*}
where  $\a^n =\{\a_k^n\}_{k=1}^n := \{\b_\ell((k-\ell-1)/\ell)\}_{k=1}^n$.
We have denoted by $P_{\a^n}$ the distribution of $X$ such that 
$E[X_k] = \a_k^n$ for $1\le k \le n$,
by $\nu_{\b_\ell(\cdot)}^\Ga$
the measure $\nu_{\b_\ell(\cdot)}^{\La_\ell}$ restricted on $\Ga$, and
$(\t_{\ell+1}\Ga)^c = \{1,2,\ldots,n\}\setminus \t_{\ell+1}\Ga$.
However, under the scaling conditions \eqref{eq:4.E.1} and \eqref{eq:4.E.2},
by Lemma \ref{lem:4.2}, $\b_\ell(\cdot)$ converges to $\b(\cdot)$ 
uniformly: 
\begin{equation} \label{eq:4.c}
\lim_{\ell\to\infty}\max_{x\in [-1,1]} |\b_\ell(x)-\b(x)|=0.
\end{equation}
Hence, we have that
$$
\sum_{\xi\in\{0,1\}^\Ga} \big(\prod_{j=1}^p f_j(\xi_j)- 
\prod_{j=1}^p E_{\nu_{\b(x_j)}}[f_j]\big) \nu_{\b_\ell(\cdot)}^\Ga(\xi) 
=o(1),
$$
as $\ell\to\infty$.  Therefore, \eqref{eq:4.E.8} can be rewritten as
\begin{equation} \label{eq:4.E.A}
= \sum_{\xi\in\{0,1\}^\Ga} \big(\prod_{j=1}^p f_j(\xi_j)- 
\prod_{j=1}^p E_{\nu_{\b(x_j)}}[f_j]\big) 
\nu_{\b_\ell(\cdot)}^\Ga(\xi) \times \left\{ \frac{B_\ell}{A_\ell} -
1 \right\} + o(1),
\end{equation}
where
\begin{align*}
& A_\ell := P_{\a^n}\left( S_n^\emptyset(X) =K, 
T_n^\emptyset(X)=M+ (\ell+1) K 
  \right), \\
& B_\ell := P_{\a^n}\left( 
S_n^{\t_{\ell+1}\Ga}(X) =K - S_n^{(\t_{\ell+1}\Ga)^c}(\t_{\ell+1}^{-1}\xi),
  \phantom{\frac{n+1}2}\right.\\
& \qquad\qquad\quad  \left.
T_n^{\t_{\ell+1}\Ga}(X) =M+ (\ell+1) K 
  - T_n^{(\t_{\ell+1}\Ga)^c}(\t_{\ell+1}^{-1}\xi)\right).
\end{align*}

We now apply Proposition \ref{lem:4.E.2} to compute the asymptotic behaviors 
of $A_\ell$ and $B_\ell$.  Note that, from \eqref{eq:4.c},
$\a^n=\{\a_k^n\}_{k=1}^n$ satisfies
the condition \eqref{eq:3.alpha} with $\a(x) :=\b(2x-1), x\in [0,1]$.
By the choice of $(a_\ell,b_\ell)$, we can show that
$y_1=y_2=O(1/\sqrt{n})$ as $\ell\to\infty$ (or equivalently $n\to\infty$)
for $y_1$ and $y_2$ defined by \eqref{eq:4.E.5} for both $A_\ell$ and 
$B_\ell$.  Indeed, for $A_\ell$,
\begin{align*}
y_1 & = \frac1{\sqrt{U_n}}(K-E_n) = \frac1{\sqrt{U_n}}\left\{K-
\sum_{k=1}^n \b_\ell \left(\frac{k-\ell-1}{\ell}\right)\right\} \\
& =\frac1{\sqrt{U_n}}\left[K- \ell \left\{ \int_{-1}^1 \b_\ell(x)dx
+O \left(\frac1{\ell}\right)\right\}\right] \\
& =\frac1{\sqrt{U_n}}\left[K- \ell \left\{ \frac{2K}{2\ell+1}
+O \left(\frac1{\ell}\right)\right\}\right]
= O(1/\sqrt{n}).
\end{align*}
Here, for the third equality, we have used the uniform bound:
$\sup_\ell \max_{x\in [-1,1]} |\b_\ell'(x)|<\infty$, recall the
choice of $\b_\ell(\cdot)$ for the fourth and \eqref{eq:3.0} for the
last.  Similarly,
\begin{align*}
y_2 & = \frac1{\sqrt{V_n}}(L-F_n) = \frac1{\sqrt{V_n}}\left\{M+(\ell+1)K-
\sum_{k=1}^n k \b_\ell \left(\frac{k-\ell-1}{\ell}\right)\right\} \\
& =\frac1{\sqrt{V_n}}\left[M+(\ell+1)K- \ell^2 \left\{ \int_{-1}^1 
(x+1)\b_\ell(x)dx +O \left(\frac1{\ell}\right)\right\}\right] \\
& = O(1/\sqrt{n}).
\end{align*}
For $B_\ell$, compared with $A_\ell$, there are additional terms 
$$
-\frac1{\sqrt{U_n}} \left\{S_n^{(\t_{\ell+1}\Ga)^c}(\t_{\ell+1}^{-1}\xi)
- \sum_{k\in \t_{\ell+1}\Ga} \b_\ell \left(\frac{k-\ell-1}{\ell}\right)
\right\}
$$
in $y_1$ and 
$$
-\frac1{\sqrt{V_n}} \left\{T_n^{(\t_{\ell+1}\Ga)^c}(\t_{\ell+1}^{-1}\xi)
- \sum_{k\in \t_{\ell+1}\Ga} k \b_\ell \left(\frac{k-\ell-1}{\ell}\right)
\right\}
$$
in $y_2$, but both these terms behave as $O(1/\sqrt{n})$.
Thus, from Proposition \ref{lem:4.E.2}, two terms $A_\ell$ and $B_\ell$ 
both behave as
\begin{equation*}
A_\ell, \; B_\ell = \frac1{\sqrt{U_n V_n}} \left\{
\frac1{2\pi\sqrt{1-\la^2}} +o(1)\right\},
\end{equation*}
as $\ell\to\infty$ (or equivalently $n\to\infty$).  Since $1-\la^2 >0$, this
shows that $A_\ell/B_\ell \to 1$ and completes the poof of Theorem 
\ref{thm:4.E.1} from \eqref{eq:4.E.A}.

\begin{rem} \label{rem:4.1}
If we take $\b(\cdot)$ itself instead of $\b_\ell(\cdot)$, the terms $y_1$
and $y_2$ in the exponential of $q_0$ behave as $o(\sqrt{n})$ so that
$q_0$ in general converges to $0$ and the local limit theorem turns 
out to be useless.
\end{rem}

\section{Proof of Corollary \ref{cor:2.2} and relations to Vershik curves}

For the proof of Corollary \ref{cor:2.2}, it suffices to show the weaker
convergence
\begin{equation}\label{eq:5.a}
\lim_{\ell\to\infty} \nu_{\La_\ell,K,M}\left( \left| \lan \tilde{\psi}^\ell,
\fa \ran - \lan \psi, \fa\ran\right| > \de \right) =0,
\end{equation}
for every $\de>0$ and $\fa\in C([-1,1])$, where $\lan \psi, \fa\ran=
\int_{-1}^1\psi(x)\fa(x)dx$.  In fact, \eqref{eq:5.a} implies the stronger
convergence result stated in Corollary \ref{cor:2.2} due to the
monotonicity of $\tilde{\psi}^\ell$, see Remark 2.5 in \cite{FS-1}. 
For proving \eqref{eq:5.a}, it is enough to show that
$\lan \tilde{\psi}^\ell,\fa \ran$ converges to $\lan \psi, \fa\ran$
in $L^2$-sense.  However, recalling \eqref{eq:2.a-2} and \eqref{eq:2.a-1},
a simple computation leads to
$$
\lan \tilde{\psi}^\ell, \fa\ran = \frac1{\ell} \sum_{k=-\ell+1}^\ell
  \eta_k \tilde{\fa}\left(\frac{k}{\ell}\right),
$$
where $\tilde{\fa}(x) = \int_{-1}^x \fa(y)dy$.  Therefore, once we can show
that
\begin{equation} \label{eq:5.1}
\lim_{\ell\to\infty} E_{\nu_{\La_\ell,K,M}}
\left[ \frac1{\ell} \sum_{k=-\ell+1}^\ell \eta_k \tilde{\fa}
\left(\frac{k}{\ell}\right) \right] = \lan\b,\tilde{\fa}\ran,
\end{equation}
and
\begin{equation} \label{eq:5.2}
\lim_{\ell\to\infty} E_{\nu_{\La_\ell,K,M}}
\left[ \left\{\frac1{\ell} \sum_{k=-\ell+1}^\ell \eta_k \tilde{\fa}
\left(\frac{k}{\ell}\right) \right\}^2\right] = \lan\b,\tilde{\fa}\ran^2,
\end{equation}
the $L^2$-convergence follows by noting that $\lan\b,\tilde{\fa}\ran
= \lan\psi,\fa\ran$.  We only give the proof of \eqref{eq:5.2}, since
\eqref{eq:5.1} is similar and easier.  To this end, from the estimate
$$
\left| \lan \tilde{\psi}^\ell, \fa\ran - \int_{-1}^1
  \eta_{[\ell x]+1} \tilde{\fa}(x)dx \right| \le 2 \sup_{|x_1-x_2|\le 
  \frac1{\ell}} |\tilde{\fa}(x_1)-\tilde{\fa}(x_2)|,
$$
it is enough to prove
\begin{equation} \label{eq:5.3}
\lim_{\ell\to\infty} E_{\nu_{\La_\ell,K,M}}
\left[ \left\{\int_{-1}^1 \eta_{[\ell x]} \tilde{\fa}(x)dx
\right\}^2\right] = \lan\b,\tilde{\fa}\ran^2,
\end{equation}
where $[\ell x]$ stands for the integer part of $\ell x$.
However, the expectation in \eqref{eq:5.3} is expanded as
$$
\int_{-1}^1 \int_{-1}^1 
E_{\nu_{\La_\ell,K,M}}[\eta_{[\ell x_1]}\eta_{[\ell x_2]}]
 \tilde{\fa}(x_1)\tilde{\fa}(x_2) dx_1 dx_2,
$$
and Theorem \ref{thm:4.E.1} applied for $p=2$ and $f_1(\eta) = f_2(\eta)
=\eta_0$ implies that
$$
\lim_{\ell\to\infty} 
E_{\nu_{\La_\ell,K,M}}[\eta_{[\ell x_1]}\eta_{[\ell x_2]}]
= \b(x_1) \b(x_2).
$$
Thus, Lebesgue's convergence theorem shows \eqref{eq:5.3} and this completes
the proof of Corollary \ref{cor:2.2}.

We finally compare our result in Corollary \ref{cor:2.2} with those in
\cite{BBE}.  In \cite{BBE}, the grand canonical ensembles 
in a rectangular box for the uniform (Bose) statistics are dealt based on
a combinatorial method, while we have discussed the canonical ensembles 
in a rectangular box for the restricted uniform (Fermi) statistics due to a 
probabilistic approach.  

Theorem 1 of \cite{BBE} shows, by rotating the plane coordinates by 45 
degree, that the limit curve $t\in [0,1] \mapsto L(t)$ in the box
with the ratio of height/width given by $\bar{\rho}/(1-\bar{\rho}), \bar{\rho}
\in (0,1)$ is determined by
$$
L(t) \equiv L_{\bar{\rho},\bar{c}}(t) = \frac1{\bar{c}} \log \frac{h(t)}{h(0)},
$$
where $h(t) = e^{-\bar{c}t}-e^{\bar{c}t} + e^{-\bar{c}(2-2\bar{\rho}-t)} 
-e^{-\bar{c}(t-2\bar{\rho})}, t\in [0,1],$ and $\bar{c}\in\R$ is a parameter
which controls the area.  By rotating
back to the original coordinates, this curve is transformed to the curve
$v=\phi(u)$ in $u$-$v$ plane given implicitly by
\begin{equation} \label{eq:5.5}
u= \frac1{\sqrt{2}} (t+L(t)) \in [0,(1-\bar{\rho})/\sqrt{2}],
\quad 
v= \frac1{\sqrt{2}} (-t+L(t)) \in [-\sqrt{2}\bar{\rho},0].
\end{equation}
However, as seen in Proposition 4.4 of \cite{FS-1}, the curve $v=\phi(u)$
appearing in the uniform statistics can be related to the curve 
$y=\bar{\psi}(x)$ in $x$-$y$ plane appearing in the restricted uniform
statistics by
$$
\bar{\b}(x) \equiv -\bar{\psi}'(x) = 
\frac{-\phi'(G_{\phi}^{-1}(x))}{1-\phi'(G_{\phi}^{-1}(x))},
$$
where $G_{\phi}(u)=u-\phi(u)$ and $G_{\phi}^{-1}$ is its inverse function.
But, \eqref{eq:5.5} shows that $G_{\phi}(u)=\sqrt{2}t$ and thus, setting
$x=G_{\phi}(u) \in [0,\sqrt{2}]$, we have that
$$
\bar{\b}(\sqrt{2}t) = \frac{-\phi'(u)}{1-\phi'(u)} = \frac{1-L'(t)}2,
$$
or equivalently
$$
\bar{\b}(x) = \frac{1-L'(x/\sqrt{2})}2.
$$
A simple computation leads to $L'(t) = h'(t)/(\bar{c}h(t))$ and
$L''(t) = \bar{c}(1-L'(t)^2),$ since $h''(t) = \bar{c}^2 h(t)$.  This proves
that
$$
\bar{\b}'(x) = -\frac1{2\sqrt{2}} L''(x/\sqrt{2})
= \frac{\bar{c}}{2\sqrt{2}} \left(L'(x/\sqrt{2})^2-1\right)
= - \sqrt{2}\bar{c} \, \bar{\b}(x) (1-\bar{\b}(x)),
$$
so that we arrive at the ordinary differential equation for $\bar{\psi}$:
\begin{equation} \label{eq:5.6}
\bar{\psi}''(x) + \sqrt{2}\bar{c} \, \bar{\psi}'(x) (1+\bar{\psi}'(x)) =0, 
\quad x\in [0,\sqrt{2}].
\end{equation}
Moreover, noting that $L(0)=0$ and $L(1)=1-2\bar{\rho}$, the height 
difference of $\bar{\psi}$ at two boundary points is given by
$$
\bar{\psi}(0)- \bar{\psi}(\sqrt{2}) = \int_0^{\sqrt{2}}\bar{\b}(x)dx
= \sqrt{2}\bar{\rho},
$$
or, by normalizing $\bar{\psi}(\sqrt{2})=0$, we have that $\bar{\psi}(0)
=\sqrt{2}\bar{\rho}$.

To compare $\bar{\psi}$ with $\psi$ in Corollary \ref{cor:2.2}, note that
$\psi$ is defined for $x\in [-1,1]$.  The shift in $x$ does not change
the form of the equation.  To consider on the interval of same length,
we introduce the scaling for $\bar{\psi}$ defined on $[0,\sqrt{2}]$ by
$$
\bar{\psi}^\ga(x) := \frac1{\ga} \bar{\psi}(\ga x), 
  \quad x\in [0,\sqrt{2}/\ga],
$$
for $\ga>0$.  Then, if $\bar{\psi}$ satisfies $\bar{\psi}''+\k \bar{\psi}'
(1+\bar{\psi}')=0$, $\bar{\psi}^\ga$ satisfies the equation
$(\bar{\psi}^\ga)''+\k\ga (\bar{\psi}^\ga)'(1+(\bar{\psi}^\ga)')=0$.
Applying this for the equation \eqref{eq:5.6} with $\ga=
1/\sqrt{2}$ and $\k=\sqrt{2} \bar{c}$, we can derive the equation for 
$\bar{\psi}^{1/\sqrt{2}}$:
$$
(\bar{\psi}^{1/\sqrt{2}})''+\bar{c} (\bar{\psi}^{1/\sqrt{2}})'
(1+(\bar{\psi}^{1/\sqrt{2}})')=0, \quad x\in [0,2],
$$
with $\bar{\psi}^{1/\sqrt{2}}(0)=2\bar{\rho}$ and 
$\bar{\psi}^{1/\sqrt{2}}(2)=0$.  This coincides with the ordinary differential
equation \eqref{eq:Vershik} for the Vershik curve with $c=\bar{c}$ and 
$\rho=\bar{\rho}$.  Thus, we can identify the limit curves
in a rectangular box for grand canonical ensembles in the uniform statistics
and for canonical ensembles in the restricted uniform statistics, namely
$\bar{\psi}^{1/\sqrt{2}} =\psi$ (except the shift in $x$ by $1$), if the 
relations $\bar{c} = -b$ and $\bar{\rho}=\rho$ hold.

\vskip 7mm
\noindent
{\bf Acknowledgement} $\,$ 
The author thanks Yoshiki Otobe for informing him the reference \cite{GR}.

\end{document}